\documentclass{amsart}

\usepackage{amsmath,amssymb,amsfonts,mathrsfs}
\usepackage{graphicx}

\newtheorem{Theo}{Theorem}[section]

\newtheorem{Cor}[Theo]{Corollary}

\newtheorem{Prop}[Theo]{Proposition}

\newcommand{\D}{\mathbb D}

\newcommand{\R}{\mathbb R}
\newcommand{\N}{\mathbb N}
\newcommand{\C}{\mathbb C}

\newcommand{\UU}{\mathscr U}
\newcommand{\BB}{\mathscr B}

\newcommand{\eps}{\varepsilon}

\newcommand{\gibt}[1]{\exists\,#1\;}
\newcommand{\alle}[1]{\forall\,#1\;}

\newcommand{\<}{\langle}
\renewcommand{\>}{\rangle}

\DeclareMathOperator{\Kern}{Kern}

\begin{document}

\title[Diametral dimensions of Fr\'echet spaces]{Diametral dimensions of Fr\'echet spaces}

\author{Lo\"ic Demeulenaere}
\thanks{L.\ Demeulenaere has a grant from the ``Fonds National de la Recherche Scientifique'' (FRIA grant)}
\address{Universit\'e de Li\`ege, D\'epartement de Math\'ematiques, B37, 4000 Li\`ege, Belgium }
\email{Loic.Demeulenaere@ulg.ac.be}

\author{Leonhard Frerick}
\address{Universit\"at Trier, FB IV -- Mathematik, 54286 Trier, Germany}
\email{frerick@uni-trier.de}

\author{Jochen Wengenroth}
\address{Universit\"at Trier, FB IV -- Mathematik, 54286 Trier, Germany}
\email{wengenroth@uni-trier.de}

\dedicatory{Dedicated to the memory of Tozun Terzio\v{g}lu} 



\begin{abstract}
The diametral dimension is an important topological invariant in the category of Fr\'echet spaces which has been used, e.g., to distinguish types of Stein manifolds. 
We introduce variants of the classical definition in order to solve an old conjecture of Bessaga, Mityagin, Pe{\l}czy\'nski, and Rolewicz at least for nuclear Fr\'echet spaces.
Moreover, we clarify the relation between an invariant recently introduced by Terzio\v{g}lu and the by now classical condition $(\overline \Omega)$ of Vogt and Wagner.
\end{abstract}

\maketitle

\section{Kolmogorov widths and diametral dimensions}

Kolmogov widths (or diameters) are a quantitative measure for compactness in normed spaces: for absolutely convex sets $V$ and $U$ of a vector space $X$
(typically $U$ is the unit ball of a given norm) and $n\in \N_0$ the $n$-th width is
\[
  \delta_n(V,U)=\inf\{\delta>0: V\subseteq \delta U+L \text{ for an at most $n$-dimensional subspace $L$ of $X$}\}
\]
(the dependence on $X$ is notationally surpressed). If $V$ is bounded with respect to $U$ (i.e., $\delta_0(V,U)<\infty$)
then $V$ is precompact with respect to the
the Minkowski functional of $U$ (which is a seminorm with unit ball $U$) if and only if $\delta_n(V,U)\to 0$.
This elementary fact is proposition 1.2 in Pinkus' book \cite{Pi85} where much more information about $n$-widths can be found.

For a locally convex space (l.c.s) $X$ with the system $\UU_0(X)$ of absolutely convex $0$-neighbourhoods the {\em diametral dimension} of $X$ is the sequence space
\[
  \Delta(X)=\{(\xi_n)_{n\in\N_0}\in \R^{\N_0}: \alle{U\in\UU_0(X)} \gibt{V\in\UU_0(X)} \text{ with } \xi_n\delta_n(V,U) \to 0\}.
\]
This space is a topological invariant, i.e., if $X$ and $Y$ are isomorphic l.c.s.\ then $\Delta(X)=\Delta(Y)$. Even more, if $Y$ is a quotient of $X$ then 
$\Delta(X)\subseteq \Delta(Y)$. Moreover, $X$ is a Schwartz space (i.e., every $0$-neighbourhood $U$ contains another one which is precompact with respect to $U$) 
if and only if $\ell^\infty\subseteq\Delta(X)$.

There are several versions of diametral dimensions of locally convex spaces in the literature, all going back to an idea of Pe{\l}czy\'nski 
\cite{Pel57} from 1957. The formulation above can be found in \cite{Mi61} where Mityagin refers to a joint work with  Bessaga, Pe{\l}czy\'nski, and Rolewicz 
which, to our best knowledge, eventually did not appear in print. 

Implicitely, Mityagin also considered the following variant
\[
  \Delta_b(X)=\{\xi \in \R^{\N_0}: \alle{U\in\UU_0(X)} \alle{B\in\BB(X)} \text{ we have } \xi_n\delta_n(B,U) \to 0\}.
\]
where $\BB(X)$ is the system of all bounded subsets of $X$. The obvious property $\delta_n(B,U) \le S \delta_n(V,U)$ for $B\subseteq SV$ implies
\[
  \Delta(X) \subseteq \Delta_b(X)
\]
for all l.c.s., and all bounded sets of $X$ are precompact if and only if $\ell^\infty\subseteq \Delta_b(X)$.
Referring to the joint work mentioned above, \cite[Proposition 9]{Mi61} claims $\Delta_b(X)=\Delta(X)$ for all Fr\'echet spaces $X$
(actually, $\xi_n\delta_n(B,U)\to 0$ is only required for compact sets $B$ but then the statement is clearly wrong even for Banach spaces).

It was probably soon realized that $\Delta=\Delta_b$ cannot be true in full generality, since the famous Grothendieck-K\"othe example \cite[27.21]{MV} of a Fr\'echet-Montel space $X$ which 
is not Schwartz satisfies
$\Delta(X)=c_0$ and $\Delta_b(X)\supseteq \ell^\infty$. The article \cite{Ter13} of Terzio\v{g}lu describes this explicitely but without definite conclusion when the equality 
is in fact true. The most optimistic conjecture is that $\Delta(X) = \Delta_b(X)$ holds for all Fr\'echet-Schwartz spaces (for spaces which are not Montel we trivially have 
$\Delta(X) = \Delta_b(X)=c_0$).

The main result of this part of the paper is a proof for hilbertizable Fr\'echet-Schwartz spaces, i.e., the seminorms can be given by (semi-) scalar products,
in particular, this includes the important case of nuclear Fr\'echet spaces.

The following variants of the diametral dimensions turn out to be very useful for this purpose:
\begin{align*}
  \Delta^\infty(X)&=\{\xi\in \R^{\N_0}: \alle{U\in\UU_0(X)} \gibt{V\in\UU_0(X)} \text{ with } \xi_n\delta_n(V,U) \text{ bounded}\},\\
  \Delta_b^\infty(X)&=\{\xi\in \R^{\N_0}: \alle{U\in\UU_0(X)} \alle{B\in\BB(X)} \text{  } \xi_n\delta_n(B,U)  \text{ is bounded}\}.
\end{align*}
For every locally convex space $X$
we have the obvious inclusions: 
\begin{large}
\begin{center}
\begin{tabular}{c c c c c c c c c c c c c}
$\Delta^\infty(X)$ & $\subseteq$ &$\Delta_b^\infty(X)$  \\
    \rotatebox[origin=c]{90}{$\subseteq$}  & & \rotatebox[origin=c]{90}{$\subseteq$} \\
$\Delta(X)$& $\subseteq$&$\Delta_b(X).$
\end{tabular}
\end{center}
\end{large}
Although we are interested whether equality holds in the lower row we first consider the upper row:
%
%\[
%  \Delta(X) \subseteq \Delta^\infty(X)\subseteq \Delta_b^\infty(X) \text{ and } \Delta(X) \subseteq \Delta_b(X) \subseteq \Delta_b^\infty(X).
%\]

\begin{Prop}\label{DI=DIB}
$\Delta^\infty(X)=\Delta_b^\infty(X)$ holds for every Fr\'echet-Schwartz space $X$.
\end{Prop}

\begin{proof} Take $\xi \in\Delta^\infty_b(X)$, a $0$-neighbourhood $U$ in $X$, and assume that
$\xi_n\delta_n(U_\ell,U)$ is unbounded for all $\ell$ where $U_\ell$ is a decreasing basis of $\UU_0(X)$. Then we find a strictly increasing sequence
$n_\ell$ such that $|\xi_{n_\ell}| \delta_{n_\ell} (U_\ell,U) >\ell$ for all $\ell\in\N$.

Since $X$ is Schwartz we may assume that all $U_\ell$ are precompact with respect to $U$. Hence, there are finite sets $F_\ell \subseteq U_\ell$ with 
$$
U_\ell \subseteq  \frac{1}{|\xi_{n_\ell}|} U +F_\ell.
$$
Then $B=\bigcup_\ell F_\ell$ is bounded in $X$ because, for each $p\in\N$, all but the finitely many elements of $F_0\cup\cdots\cup F_{p-1}$ belong to $U_p$.
Since $\xi\in\Delta^\infty_b(X)$ there is a constant $C$ such that
$|\xi_{n_\ell}| \delta_{n_\ell}(B,U) <C$ for all $\ell$. 

For fixed $\ell > C+1$ there is, by definition of $\delta_n(B,U)$, an at most $n_\ell$-dimensional subspace $L$ with
$B\subseteq \frac{C}{|\xi_{n_\ell}|}  U +L$ which implies 
$$
U_\ell \subseteq  \frac{1}{|\xi_{n_\ell}|} U +F_\ell \subseteq \frac{1}{|\xi_{n_\ell}|} U + \frac{C}{|\xi_{n_\ell}|} U +L 
\subseteq \frac{C+1}{|\xi_{n_\ell}|} U +L.
$$
Hence the contradiction $\delta_{n_\ell}(U_\ell,U) \le (C+1)/{|\xi_{n_\ell}|}$. 
\end{proof}

The proposition is trivially true for Fr\'echet spaces which are not Montel (because then $\Delta^\infty(X)=\Delta_b^\infty(X)=\ell_\infty$). It thus holds for 
all {\em quasinormable} Fr\'echet spaces. Up to the change from $c_0$ to $\ell_\infty$ in definition of the diametral dimension this safes
\cite[proposition 1]{Ter13}.

In view of the proposition and the trivial inclusions mentioned above it is enough to investigate whether $\Delta(X)=\Delta^\infty(X)$ holds for Fr\'echet-Schwartz spaces.

If $V$ and $U$ are the unit balls of semi-norms $p\ge q$ the Kolmogorov widths describe approximation properties of the inclusion $(X,p)\hookrightarrow (X,q)$ and it is easy to see 
(and well-known) that
we may pass to the Hausdorff completions of these spaces, that is, $\delta_n(V,U)=\delta_n(T(B_p),B_q)$, where $B_p$ is the unit ball of the completion $X_V$ of $(X,p)/\Kern(p)$ 
and $T$ is the canonical map $X_V\to X_U$ induced by the inclusion. $X_V$ is called the {\it local Banach space} corresponding to $V$. 
If $p$ is induced by a (semi-) scalar product we call it 
the {\it local Hilbert space}.

For an operator $T:X\to Y$ between Banach spaces with unit balls $B_X$ and $B_Y$
we abbreviate $\delta_n(T)=\delta_n(T(B_X),B_Y)$. The velocity of convergence of this sequence is then a measure for the compactness of $T$.

For a scalar sequences  $(\delta_n)_{n\in\N_0}$ we write, as usual,
$$
o(\delta_n)=\{\xi\in \R^{\N_0}: \alle{\eps>0} \gibt{N\in\N} \alle{n\ge N:} |\xi_n|\le \eps |\delta_n|\}.
$$

We now consider the plausible statement that the product (composition) of two compact operators is ''strictly more compact`` than each factor. Unfortunately, we can only prove this
for Hilbert spaces.

\begin{Prop}\label{composition}
   Let $X,Y,Z$ be Hilbert spaces and $T:X\to Y$  and $S:Y\to Z$ be compact operators. Then
$$
\delta_n(S\circ T) \in o(\delta_n(S)) \text{ and } \delta_n(S\circ T) \in o(\delta_n(T)).
$$
\end{Prop}

\begin{proof}
The advantage of the Hilbert space setting is that the Kolmogorov widths coincide with the singular numbers of the compact operator, see, e.g., \cite[chapter IV]{Pi85} or
\cite{Vo00}.
Let thus 
$$
S=\sum\limits_{k=0}^\infty s_k \< \cdot,e_k\> f_k
$$ 
be a Schmidt representation of $S$ with orthonormal systems $(e_k)_k$ and
$(f_k)_k$ in $Y$ and $Z$, respectively, and the decreasing sequence of singular numbers $s_k=\delta_k(S)$. 
Taking the span of $f_0,\ldots,f_{n-1}$ as a candidate for the infimum in the definition of $\delta_n(S\circ T)$ we get
\begin{align*}
  \delta_n(S\circ T)^2 & \le \sup\left \{\left\| \sum_{k=n}^\infty s_k\<T(x),e_k\> f_k\right\|_Z^2: \|x\|_X\le 1\right\} \\
  &=  \sup\left \{\sum_{k=n}^\infty s_k^2 |\<T(x),e_k\>|^2: \|x\|_X\le 1\right\} \\
  & \le s_n^2 \sup\left \{\sum_{k=n}^\infty  |\<y,e_k\>|^2: y \in K\right\} 
\end{align*}
where $K=\overline{T(B_X)}$ is compact in $Y$. The sequence of functions 
$$
r_n(y)= \left(\sum_{k=n}^\infty  |\<y,e_k\>|^2\right)^{1/2} = \|\pi_n(y)\|_Y
$$
with the orthogonal projections $\pi_n$ onto the closed spans of 
$\{e_k: k\ge n\}$ is equicontinuous and converges pointwise to $0$. It therefore converges uniformly to $0$ on the compact set $K$.
This proves $\delta_n(S\circ T) \in o(\delta_n(S))$.

The other assertion then follows by duality since 
\[
  \delta_n(S\circ T) = \delta_n((S \circ T)^*) = \delta_n(T^*\circ S^*) \in o(\delta_n(T^*)) = o(\delta_n(T)). \qedhere
\]
\end{proof}

\begin{Prop}\label{D=DI}
$\Delta(X)=\Delta^\infty(X)$ holds for every hilbertizable Schwartz space.
\end{Prop}

\begin{proof} In the definitions of the diametral dimensions we may replace $\UU_0(X)$ by the system of all $0$-neighbourhoods which are unit balls of semi-norms induced by scalar products.
   Given $\xi\in \Delta^\infty(X)$ and $U\in\UU_0(X)$ we choose $V\in \UU_0(X)$ such that the canonical map $T:X_V\to X_U$ between the local Hilbert spaces is
compact and then we choose $W\in \UU_0(X)$ such that $\xi_n\delta_n(W,V)$ is bounded and $S:X_W \to X_V$ is compact. The previous proposition then implies 
$\xi_n \delta_n(W,U) \in o(\xi_n\delta_n(W,V))$, hence $\xi_n \delta_n(W,U)\to 0$.
\end{proof}

We remark that either statement in proposition \ref{composition} for Banach instead of Hilbert spaces 
(i.e., the product of two compact operators between Banach spaces is ''strictly more compact`` than at least one of the factors) 
would give \ref{D=DI} for all Schwartz spaces.

Combining \ref{DI=DIB} and \ref{D=DI} we can now confirm the claim of Bessaga, Mityagin, Pe{\l}czy\'nski, and Rolewicz at least for hilbertizable Fr\'echet-Schwartz spaces:

\begin{Theo}\label{theo}
 $\Delta(X)=\Delta_b(X)=\Delta^\infty(X)=\Delta_b^\infty(X)$ holds for every hilbertizable Fr\'echet-Schwartz space and, in particular, for every nuclear Fr\'echet space.
\end{Theo}

We do not know if $\Delta_b(X)=\Delta_b^\infty(X)$ holds for all Fr\'echet-Schwartz spaces. However, the same method as in \ref{D=DI} gives this for hilbertizable Fr\'echet-Montel spaces 
(note that the Grothendieck-K\"othe example can be chosen hilbertizable so that the following statement is not contained in \ref{theo}):

\begin{Prop}
$\Delta_b(X)=\Delta_b^\infty(X)$ holds for every hilbertizable Fr\'echet-Montel space.  
\end{Prop}

\begin{proof}
 We use that a hilbertizable Fr\'echet space has a fundamental system of bounded sets consisting of unit balls of Hilbert spaces
 (see, e.g., the end of the proof of \cite[29.16]{MV}) and that Fr\'echet spaces satisfy the 
 {\it strict Mackey condition} introduced by Grothendieck, i.e., for every bounded set $B$ there is an absolutely convex
 bounded set $D\supseteq B$ whose Minkowski functional induces on $B$ the topology of $X$, see \cite[theorem 5.1.27]{PCB}. 
 In our case we find thus for each bounded Hilbert ball $B$ another bounded Hilbert ball 
 $D$ such that span$(B) \hookrightarrow$ span$(D)$ is a compact inclusion between Hilbert spaces. For $U\in \UU_0(X)$ and the associated local Hilbert space $X_U$ we can thus apply
 proposition \ref{composition} to span$(B) \hookrightarrow$ span$(D) \to X_U$ and obtain $\delta_n(B,U)\in o(\delta_n(D,U))$.
\end{proof}

\section{Prominent sets}

In his last publication \cite{Ter13} T.\ Terzio\v{g}lu called a bounded subset $B$ of a l.c.s.\ $X$ {\it prominent} if
\[
  \Delta(X)=\{\xi\in\R^{\N_0}: \xi_n \delta_n(B,U)\to 0 \text{ for all } U\in\UU_0(X)\}.
\]
Having a prominent set is obviously a topological invariant, and since the right hand side above contains $\Delta_b(X)$ we have 
$\Delta(X)=\Delta_b(X)$ for all l.c.s.\ with a prominent bounded set.

By an elegant application of Grothendieck's factorization theorem Terzio\v{g}lu \cite[proposition 3]{Ter13} proved that an absolutely convex bounded set $B$ 
of a Fr\'echet space is prominent if and only, for every $U\in\UU_0(X)$, there are $V\in \UU_0(X)$ and $C\ge 0$ such that $\delta_n(V,U)\le C\ \delta_n(B,V)$ for all $n\in\N_0$.

Moreover, he proved the existence of prominent bounded sets in so-called $G_1$ spaces which form a class of K\"othe sequence spaces containing power series spaces of finite type,
and he showed that power series spaces of infinite type do not have prominent sets.

There are some more topological invariants distinguishing power series spaces of finite and infinite type, in particular condition $(\overline \Omega)$ of Vogt and Wagner, 
see \cite[chapter 29]{MV}.
We recall that a Fr\'echet space $X$ with fundamental sequence of semi-norms $\|\cdot\|_k$ and corresponding dual norms $\|\cdot\|_k^*$ satisfies $(\overline \Omega)$ if
\[
  \alle{k\in\N} \gibt{\ell\ge k} \alle{m\ge \ell} \gibt{C>0} \text{ with } \left(\|\varphi\|_\ell^*\right)^2\le  C\, \|\varphi\|_k^* \|\varphi\|_m^* \text{ for all } \varphi\in X'.
\]

\begin{Theo}\label{omegaprominent}
Every Fr\'echet space $X$ with  $( \overline{\Omega})$ has a prominent set.
\end{Theo}

\begin{proof}
 A combination of lemmas 29.13 and 29.16 in \cite{MV} shows that 
there exists a bounded Banach disk $B$ of $X$ such that, for every $U\in \UU_0(X)$, there exist $V\in \UU_0(X)$ and $C>0$ with 
\[
V \subseteq r U + \frac{C}{r} B \text{ for all $r>0$}.
\]
 Fix $U\in\UU_0(X)$ and take $V$ and $C$ as above.
Let $n \in \N_0$ and $\delta > \delta_n (B,V)$. Then, there exists an at most $n$-dimensional subspace
$L$ with $B \subseteq \delta V + L$. For $r = 2 C \delta$, we obtain
\[
V \subseteq r U + \frac{C}{r} B \subseteq r U + \frac{C \delta}{r} V + L = 2 C \delta U + \frac{1}{2}V + L.
\]
Inserting this inclusion into its right hand side, we get 
\[
V \subseteq 2 C \delta U + \frac{1}{2} \left ( 2 C \delta U + \frac{1}{2} V + L \right ) + L \subseteq 2 C \delta \left ( 1 + \frac{1}{2} \right ) U + \frac{1}{4} V + L.
\]
By iteration, for every $j \in \N$, this implies 
\[
V \subseteq 2 C \delta \left ( 1 + \frac 12 + \cdots + \frac{1}{2^{j-1}} \right ) U + \frac{1}{2^j}V + L \subseteq 4 C \delta U + \frac{1}{2^j} V + L.
\]
Choosing $j$ such that $\frac{1}{2^j}V \subseteq C \delta U$ we get
\[
V \subseteq 5 C \delta U + L.
\]
We have shown $\delta_n (V,U) \leq 5 C \delta$ and thus
\[
\delta_n (V,U) \leq 5 C \delta_n (B,V).
\]
Hence $B$ is a prominent set of $E$.
\end{proof}

We will next show that the implication in this theorem is in fact a characterization for {\it regular} K\"othe spaces: We recall that for a matrix
$A=(a_k(n))_{(k,n)\in\N_0^2}$ of positive weights with $a_k(n)\le a_{k+1}(n)$ the corresponding K\"othe space (of order $1$ -- but everything below holds for all other orders) is
\[
  \lambda_1(A)=\{(x_n)_{n\in\N_0}\in\C^{\N_0}: \|x\|_k=\sum_{n=0}^\infty a_k(n)|x_n| <\infty \text{ for all $k\in\N_0$}\}.
\]
The space and the matrix are called {\it regular} if $n\mapsto a_{k}(n)/a_{k+1}(n)$ is decreasing for every $k\in\N_0$. The advantage of regularity is 
that the Kolmogorov widths are then very easy to calculate: for the unit balls $U_k$ of the semi-norms $\|\cdot\|_k$ and $\ell\ge k$ we have e.g.\ by\cite{Ter08}
\[
\delta_n(U_\ell,U_k)= \frac{a_k(n)}{a_\ell(n)} \text{ for all $n\in\N_0$.}  
\]
More information about diametral dimensions of K\"othe spaces can be found in \cite{Ter08,BaDe16}.

A characterization of $( \overline{\Omega})$ for K\"othe spaces is well-known, e.g., \cite[Satz1.10]{Wa80}, and in fact
easily obtained from $\|\pi_n\|_\ell^*=1/a_\ell(n)$ (where $\pi_n(x)=x_n$) and the definition of $( \overline{\Omega})$: 

$\lambda_1 (A)$ has  $( \overline{\Omega})$ if and only if, for every 
$k \in \N_0$, there exists $\ell \ge k$ such that for every $m \ge \ell$ there is  $c>0$ with
\[
a_\ell^2 (n) \geq c \ a_m (n) a_k (n) \text{ for all $n \in \N_0$.}
\]

\begin{Prop}\label{regular}
  A regular K\"othe space $\lambda_1(A)$ has a prominent set if and only if it satisfies $( \overline{\Omega})$.
\end{Prop}

\begin{proof}
  We still have to show necessity of $( \overline{\Omega})$ given a prominent bounded set $B$. As supersets of prominent sets are prominent we may assume that
  $B=\bigcap_{k\in\N_0} r_kU_k$ for a sequence of scalars $r_k>0$. For $k\in\N_0$, Terzio\v{g}lu's characterization mentioned above for $U=U_k$ yields $V=U_\ell$
  for some $\ell \ge k$ and $C>0$ such that 
  \[
    \delta_n(U_\ell,U_k)\le C\delta_n(B,U_\ell).
   \]
   Given $m\ge \ell$ the inclusion $B\subseteq r_mU_m$ then yields
  \[
    \delta_n(U_\ell,U_k) \le C\delta_n(B,U_\ell)\le Cr_m \delta_n(U_m,U_\ell).
  \]
The formula for $\delta_n(U_\ell,U_k)$ from above thus gives
\[
\frac{a_k(n)}{a_\ell(n)} \le Cr_m \frac{a_\ell(n)}{a_m(n)} \text{ for all $n\in\N_0$}
\]
which implies $(\overline{\Omega})$.
\end{proof}

We will finally show that products of two power series spaces of different type like $H(\D)\times H(\C)$ may have prominent bounded sets.
This shows that having prominent bounded sets is not inherited by complemented subspaces. 

We recall that for an increasing sequence $0<\alpha_n\to\infty$ the power series spaces of finite type $\Lambda_0(\alpha)=\lambda_1(A)$ and infinite type
$\Lambda_\infty(\alpha)=\lambda_1(B)$ are regular K\"othe spaces corresponding to the matrices 
\[
   A=\left( \exp\left(-\alpha_n/k\right) \right)_{(k,n) \in \N_0^2} \text{ and } B=\left( \exp\left(k\alpha_n\right) \right)_{(k,n) \in \N_0^2}.
\]
The sequence $\alpha$ is called {\it stable} if $\alpha_{2n}/\alpha_n$ is bounded (this characterizes that $\Lambda_r(\alpha)$ is isomorphic to its cartesian square).

\begin{Prop}
  For every stable sequence $\alpha_n\to\infty$ the space $\Lambda_0(\alpha)\times \Lambda_\infty(\alpha)$ has a prominent bounded set but does not satisfy 
  $( \overline{\Omega})$.
\end{Prop}

\begin{proof}
  As $\Lambda_0(\alpha)$ satisfies $( \overline{\Omega})$ it has a prominent bounded set $B$, and we will show that $B\times \{0\}$ is prominent in
   $\Lambda_0(\alpha)\times \Lambda_\infty(\alpha)$.
   
  We write $U_k^0$ and $U_k^\infty$ for the canonical $0$-neighbourhoods in  $\Lambda_0(\alpha)$ and $\Lambda_\infty(\alpha)$, respectively, as well as
  $\delta_n^0$, $\delta_n^\infty$, and $\delta_n^\times$ for the Kolmogorov diameters in  $\Lambda_0(\alpha)$, $\Lambda_\infty(\alpha)$, and their product.
  
  The proof is based on the simple observation
  \[
    \delta^\times_{2n}(V^0\times V^\infty,U^0\times U^\infty) \le \max\{\delta^0_n(V^0,U^0),\delta^\infty_n(V^\infty,U^\infty)\}: 
  \]
  Indeed, if $V^r\subseteq \delta U^r +L^r$ for $r\in\{0,\infty\}$ with at most $n$-dimensional subspaces then $L^0\times L^\infty$ is an at most
  $2n$-dimensional subspace with 
  \[
    V^0\times V^\infty\subseteq \delta(U^0\times U^\infty)+ L^0\times L^\infty. 
  \]
  Since power series spaces are regular the Kolmogorov diameters $\delta_n^r$ are easily calculated, in particular, $\delta^\infty_n(V^\infty,U^\infty)$ are much smaller than
  $\delta^0_n(V^0,U^0)$. We will finally need stability to compare $\delta_n^0$ and 
  $\delta_{2n}^0$. However, the details of the proof need some care and we first give precise estimates for $\delta_{2n}^0$ which are quite easily obtained from the formula
  $\delta^0_n (U^0_k,U^0_m) = \frac{a_m(n)}{a_k(n)}=\exp((1/k - 1/m)\alpha_n) $: 
  
  \medskip
 
   For all $m \in \N_0$, there exists $s \geq m$ such that, for all $k \geq s$, we have
\[
\delta^0_n (U^0_k,U^0_m) \leq \delta^0_{2n} (U^0_k,U^0_s) \text{ for every $n \in \N_0$.}
\]

Indeed, we take $C > 0$ such that $\alpha_{2n}/\alpha_n \leq C$ for all $n \in \N$. 
If $m \in \N$ is given, we choose $s \geq (C+1)m$. Then, if $k > s$, we have
\[
\frac{\alpha_{2n}}{\alpha_n} \leq C \leq \frac sm -1\leq \dfrac{1/m - 1/k}{1/s} \leq \dfrac{1/m - 1/k}{1/s - 1/k}
\]
for every $n \in \N_0$. Therefore, $(1/s - 1/k) \alpha_{2n} \leq (1/m - 1/k) \alpha_n$ for all $k\ge s$ which
implies
\[
\delta^0_n (U^0_k,U^0_m) = \exp((1/k - 1/m)\alpha_n) \leq \exp((1/k-1/s)\alpha_{2n}) = \delta^0_{2n} (U^0_k,U^0_s). 
\]

Let now $U$ be a $0$-neighbourhood in $\Lambda_0(\alpha)\times \Lambda_\infty(\alpha)$ which we may assume to be of the form
$U=U_m^0\times U_m^\infty$ for some $m \in \N$ (because the sets $\frac 1m U^r_m$ are bases of the $0$-neighbourhood filters). 
We choose $s$ as above and then, according to Terzio\v{g}lu's characterization of prominence,
$k\ge s$ and $C\ge 0$ such that for all $n\in\N_0$
\[
  \delta^0_n( U^0_{k},U^0_{s}) \leq C \delta^0_n \left ( B,U^0_{k} \right).
\]
We then get for all $n\in\N_0$
\[
\delta^0_n( U^0_{k},U^0_{m}) \leq \delta^0_{2n}(U^0_{k},U^0_s) \leq C \delta^0_{2n} \left ( B,U^0_{k} \right).
\]
 Moreover, since $\alpha_{n+1}\le \alpha_{2n}\le C \alpha_n$ we find $\ell\in\N$ such that
\begin{align*}
\delta_n^\infty(U_\ell^\infty,U_m^\infty) & =\exp((m-\ell)\alpha_n) \leq \exp((1/k - 1/m)\alpha_{n+1})\\
  & =\delta^0_{n+1}(U^0_k,U^0_m)\le \delta^0_{n}(U^0_k,U^0_m)
\end{align*}
for all $n \in \N_0$. 
With these choices we get
\begin{align*}
&\delta^\times_{2n} \left ( U^0_{k} \times U^\infty_\ell,U^0_{m} \times U^\infty_{m} \right ) 
\leq \max \left \{ \delta^0_n ( U^0_{k},U^0_{m}), \delta^\infty_n (U^\infty_{\ell},U^\infty_{m}) \right \} \\
& =  \delta^0_n( U^0_{k},U^0_{m})
\leq C \delta^0_{2n}(B,U^0_{k}) 
= C \delta^\times_{2n}( B \times \{ 0 \}, U^0_{k} \times U^\infty_{\ell})
\end{align*}
for every $n \in \N_0$ (the last equality is immediate from the definition of Kolmogorov widths).
Similarly, we have for odd dimensions $2n+1$
\begin{align*}
&\delta^\times_{2n+1} \left ( U^0_{k} \times U^\infty_\ell,U^0_{m} \times U^\infty_{m} \right ) 
\leq \max \left \{ \delta^0_{n+1} ( U^0_{k},U^0_{m}), \delta^\infty_n (U^\infty_{\ell},U^\infty_{m}) \right \} \\
&= \delta^0_{n+1}( U^0_{k},U^0_{m})
\leq C \delta^0_{2n+2}(B,U^0_{k}) \leq C \delta^0_{2n+1}(B,U^0_{k}) \\
&= C \delta^\times_{2n+1}( B \times \{ 0 \}, U^0_{k} \times U^\infty_{\ell}).
\end{align*}
We have thus shown for all $n\in\N_0$
\[
  \delta^\times_{n} \left ( U^0_{k} \times U^\infty_\ell,U^0_{m} \times U^\infty_{m} \right ) \le C \delta^\times_{n}( B \times \{ 0 \}, U^0_{k} \times U^\infty_{\ell}),
\]
which implies that $B\times \{0\}$ is a prominent bounded set in $\Lambda_0 (\alpha) \times \Lambda_{\infty} (\alpha)$.
\end{proof}

Combined with proposition \ref{regular} we get the following application of Terzio\v{g}u's invariant of having a bounded prominent set:

\begin{Cor}
For every stable sequence $\alpha_n\to\infty$ the space $\Lambda_0(\alpha)\times \Lambda_\infty(\alpha)$ is not isomorphic to a regular K\"othe space.
\end{Cor}
 
We thank D.\ Vogt for the remark that this corollary can also be deduced from a result of Zahariuta \cite[theorem 12]{Za73}.
\bibliographystyle{amsalpha}
\bibliography{dimensions-references}
\end{document}